\documentclass[a4paper,11pt]{article}
\usepackage[T1]{fontenc}
\usepackage{lmodern}
\usepackage{amsmath,amsthm,amssymb,amsfonts}
\usepackage[dvipsnames,svgnames,table]{xcolor}
\usepackage{pgf,tikz}
\usetikzlibrary{calc}
\usetikzlibrary{arrows.meta}
\usepackage{enumitem}
\usepackage{fullpage}
\usepackage{mathtools}
\usepackage{nicematrix}
\usepackage{booktabs}
\usepackage[title]{appendix}
\usepackage{thmtools, thm-restate}
\usepackage[colorlinks=true,linkcolor=RoyalBlue,urlcolor=RoyalBlue,citecolor=PineGreen]{hyperref}
\usepackage[nameinlink]{cleveref}
\usepackage{verbatim}

\newtheorem{theorem}{Theorem}[section]
\newtheorem{proposition}[theorem]{Proposition}
\newtheorem{lemma}[theorem]{Lemma}
\newtheorem{corollary}[theorem]{Corollary}
\newtheorem{conjecture}[theorem]{Conjecture}

\theoremstyle{definition}

\theoremstyle{remark}
\newtheorem{remark}[theorem]{Remark}

\newcommand{\suchthat}{\;\ifnum\currentgrouptype=16 \middle\fi|\;}

\DeclareMathOperator{\rank}{rank}

\DeclareMathOperator{\Circ}{Circ}

\colorlet{colbg}{white}
\colorlet{colfg}{black}
\colorlet{colgraphv}{colfg!75!white}
\colorlet{colgraphe}{colfg!55!white}
\colorlet{colG}{DarkSeaGreen}
\definecolor{colR}{HTML}{CC6677}
\definecolor{colO}{HTML}{DDCC77}
\definecolor{colB}{HTML}{6699CC}
\colorlet{colY}{Gold!90!black}
\definecolor{myorange}{RGB}{239,128,108}
\definecolor{myblue}{RGB}{154,210,201}
\definecolor{mypurple}{RGB}{117,0,235}

\tikzstyle{vertex}=[fill=colgraphv,circle,inner sep=0pt, minimum size=4pt]
\tikzstyle{edge}=[line width=1.5pt,colgraphe]
\tikzstyle{dedge}=[edge, -{LaTeX[round,length=8pt]}]
\tikzstyle{oedge}=[line width=1.5pt,myorange]
\tikzstyle{bedge}=[line width=1.5pt,myblue]
\tikzstyle{pedge}=[line width=1.5pt,mypurple]

\tikzstyle{labelsty}=[font=\scriptsize]

\title{Constructing pathological symmetric abstract rigidity matroids}

\author{
Sean Dewar\thanks{Department of Computer Science, KU Leuven}
\and
Joannes Vermant\thanks{Department of Mathematics and Mathematical Statistics, Ume{\aa}  University}
}
\date{}

\begin{document}

\maketitle

\begin{abstract}
We prove that for each $d \geq 2$, there exist infinitely many symmetric abstract $d$-rigidity matroids that are not the $d$-rigidity matroid, the $d$-hyperconnectivity matroid, or the $d$-cofactor matroid.
When $d=2$, our constructed matroids are pathological, in the sense that they are not contained within any abstract 2-rigidity matroid family.
Our construction additionally provides a counter-example to a conjecture of Tyomkyn.
\end{abstract}

\section{Introduction}

\subsection{Background}

All graphs throughout the paper are assumed to be finite and simple;
i.e.~finitely many vertices, no loops and no parallel edges (and hence also finitely many edges).

Let $K_n$ be the complete graph with vertex set $[n] := \{1,\ldots,n\}$ and edge set $\binom{[n]}{2}$.
Given a positive integer $d \leq n$,
a matroid $M$ with ground set $\binom{[n]}{2}$ is an \emph{abstract $d$-rigidity matroid on $K_n$} if it satisfies the following two conditions:
\begin{enumerate}
    \item the rank of $M$ is $dn - \binom{d+1}{2}$; and
    \item every copy of $K_{d+2}$ in $\binom{[n]}{2}$ is a circuit of $M$.
\end{enumerate}
Abstract rigidity matroids were first introduced by Graver \cite{graver1991}, albeit with a slightly different definition via matroid closure;
here we have opted for the simpler definition of abstract rigidity matroids that was shown to be equivalent by Nguyen \cite{nguyen2010}.
We say that an abstract $d$-rigidity matroid $M$ on $K_n$ is \emph{symmetric} if for every pair of isomorphic graphs $G=([n],E)$ and $G'=([n],E')$, the edge set $E$ is independent in $M$ if and only if the edge set $E'$ is independent in $M$.

The unique abstract 1-rigidity matroid on $K_n$ is the graphical matroid of $K_n$ (see, \cite[pg.~357]{graver1991}),
but for $d \geq 2$ there are many different abstract $d$-rigidity matroids.
There are three important symmetric abstract $d$-rigidity matroids that exist on any $K_n$: the \emph{$d$-rigidity matroid} $\mathcal{R}_{d,n}$;
the \emph{$d$-hyperconnectivity matroid} $\mathcal{H}_{d,n}$; and 
the \emph{$d$-cofactor matroid} $\mathcal{C}_{d,n}$.
See \Cref{sec:prelim} for exact definitions of each family of matroids and their various properties.
The matroid $\mathcal{H}_{d,n}$ is distinct from both $\mathcal{R}_{d,n}$ and $\mathcal{C}_{d,n}$ for all $d \geq 2$ and $n \geq 2d+2$, and matroids $\mathcal{R}_{d,n}$ and $\mathcal{C}_{d,n}$ are distinct for all $d \geq 4$ and $n \geq 2d+4$.
However, $\mathcal{R}_{2,n} = \mathcal{C}_{2,n}$ for all values of $n$, and it is open whether or not $\mathcal{R}_{3,n} = \mathcal{C}_{3,n}$ for all values of $n$ also.
Each family $\mathcal{R}_d := (\mathcal{R}_{d,n})_{n \geq d}$, $\mathcal{H}_d := (\mathcal{H}_{d,n})_{n \geq d}$ and $\mathcal{C}_d := (\mathcal{C}_{d,n})_{n \geq d}$ forms an \emph{abstract $d$-rigidity matroid family};
a sequence $(M_n)_{n \geq N}$ of symmetric abstract $d$-rigidity matroids, here with $M_n$ on $K_n$, such that the restriction of $M_n$ to any copy of $K_{n-1}$ is isomorphic to $M_{n-1}$.

It is currently open whether there exist other abstract $d$-rigidity matroid families distinct from $\mathcal{R}_d$, $\mathcal{H}_d$ and $\mathcal{C}_d$.
In \cite[\S 6]{tyomkyn2026planerigiditymatroids}, Tyomkyn observed one method for constructing another symmetric abstract 2-rigidity matroid on $K_6$ distinct from the matroids $\mathcal{R}_{2,6}, \mathcal{H}_{2,6}, \mathcal{C}_{2,6}$:
take the matroid $\mathcal{R}_{2,6}$ and delete every copy of the triangular prism (see \Cref{fig:3prism}) from the set of bases.
He could not, however, find a 7-vertex graph that is a basis of $\mathcal{R}_{2,7}$ where this trick worked.
Because of this, he conjectured the following:

\begin{figure}
    \centering
    \begin{tikzpicture}[scale=0.8]
			\node[vertex, label=left:{$b$}] (11) at (-1,0) {};
			\node[vertex, label=left:{$a$}] (21) at (-2,1) {};
			\node[vertex, label=left:{$c$}] (31) at (-2,-1) {};
			
			\node[vertex, label=right:{$b'$}] (12) at (1,0) {};
			\node[vertex, label=right:{$a'$}] (22) at (2,1) {};
			\node[vertex, label=right:{$c'$}] (32) at (2,-1) {};
			
			\draw[edge] (11)edge(12);
			\draw[edge] (21)edge(22);
			\draw[edge] (31)edge(32);
			
			\draw[edge] (11)edge(21);
			\draw[edge] (21)edge(31);
			\draw[edge] (31)edge(11);
			
			\draw[edge] (12)edge(22);
			\draw[edge] (22)edge(32);
			\draw[edge] (32)edge(12);
		\end{tikzpicture}
    \caption{The triangular prism graph.}
    \label{fig:3prism}
\end{figure}
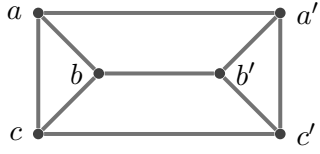

\begin{conjecture}[{\cite[Conj.~6.1]{tyomkyn2026planerigiditymatroids}}]\label{conj:tyomkyn}
    There exists a positive integer $N$ so that for every $n \geq N$, the only symmetric abstract 2-rigidity matroid on $K_n$ where every copy of $K_{3,3}$ is independent is $\mathcal{R}_{2,n}$.
\end{conjecture}

\subsection{Our contribution}

In this paper we disprove \Cref{conj:tyomkyn} by generalising Tyomkyn's construction for higher values of $d$ and $n$.
To do so, we introduce the following terminology.
For a connected graph $G=([n],E)$ and an abstract $d$-rigidity matroid $N$ on $K_n$,
we set 
\begin{equation*}
    \Circ_{G}(N) := \big\{ \text{circuits of $N$ not containing copy of $E$} \big\} \cup \big\{ \text{copies of $E$} \big\},
\end{equation*}
and we set $\mathcal{M}_G(N) := (\binom{[n]}{2},\Circ_{G}(N))$.
If $E$ is a basis of $N$ and $\Circ_{G}(N)$ is the circuit set for a matroid,
then (using the circuit axiom definition) $\mathcal{M}_G(N)$ is an abstract $d$-rigidity matroid, and the restriction of $\mathcal{M}_G(N)$ to any non-spanning subgraph $H$ agrees with the restriction of $N$ to $H$;
moreover, if $N$ is also symmetric, then so too is $\mathcal{M}_G(N)$.

The main result of the paper is the following.

\begin{theorem}\label{mainthm}
    For all integers $d \geq 2$ and $n \geq d +11$,
    there exists a graph $G=([n],E)$ for which $E$ is a basis of $\mathcal{R}_{d,n}$ and $\mathcal{M}_G(\mathcal{R}_{d,n})$ is a symmetric abstract $d$-rigidity matroid.
\end{theorem}

If $\mathcal{M}_G(\mathcal{R}_{d,n})$ is a matroid, then its restriction to $K_{n-1}$ is equal to the restriction of $\mathcal{R}_{d,n}$ to $K_{n-1}$. Thus it follows from \Cref{mainthm} that \Cref{conj:tyomkyn} is false.
\Cref{mainthm} also implies the following corollary.

\begin{corollary}\label{maincor}
    For each $d \geq 2$ and $n \geq 2d+9$, there exists a symmetric abstract $d$-dimensional rigidity matroid on $K_n$ distinct from $\mathcal{R}_{d,n}$, $\mathcal{H}_{d,n}$ and $\mathcal{C}_{d,n}$;
    moreover, the restriction of this matroid to any copy of $K_{n-1}$ is equal to $\mathcal{R}_{d,n-1}$.
\end{corollary}

As the only abstract 1-rigidity matroid on $K_n$ is the graphical matroid for $K_n$, neither \Cref{mainthm} nor \Cref{maincor} can be extended to the case $d=1$.
The values of $n$ for which the result holds can be improved, but it is unclear by how much.

Our constructed matroids are in some ways `pathological', as (in the $d=2$ case at least) they are not contained within an abstract $d$-rigidity matroid family.

\begin{corollary}\label{maincor2}
    For each $n \geq 13$, there exists a symmetric abstract 2-dimensional rigidity matroid on $K_n$ that is not contained in any abstract 2-rigidity matroid family.
\end{corollary}

\subsection{Layout of paper}

In \Cref{sec:prelim} we cover the necessary background material surrounding abstract rigidity matroids.
In \Cref{sec:main} we prove \Cref{mainthm},
which we use to then prove \Cref{maincor} and \Cref{maincor2}.
In \Cref{sec:further} we discuss what results would be required to extend \Cref{mainthm}.

\section{Abstract rigidity matroids}\label{sec:prelim}

\subsection{Rigidity matroids}\label{subsec:rigiditymatroid}

For some choice of positive integers $d,n$,
choose a \emph{generic} vector $p = (p_i)_{i \in [n]} \in (\mathbb{R}^d)^n$,
in that the coordinates of $p$ form a set with transcendence degree $dn$ over $\mathbb{Q}$.
Let $R(p)$ be the $\binom{n}{2} \times dn$ matrix with columns labelled $(i,k) \in [n] \times [d]$ ordered lexicographically and rows labelled by $\binom{n}{2}$ (with any order), where row $ij$ is of the form
\begin{equation*}
    \big[
     ~ \mathbf{0} ~ \cdots ~  \mathbf{0}  ~ \overbrace{p_i - p_j}^{i} ~ 
    \mathbf{0} ~  \cdots  ~ \mathbf{0}  ~ \overbrace{p_j - p_i}^{j} ~ 
    \mathbf{0}  ~ \cdots  ~ \mathbf{0}  ~ 
    \big]
\end{equation*}
(Here we treat each $p_v$ as a row vector and $\mathbf{0}$ as the length-$d$ all-zeroes row vector.)
We define the $d$-dimensional rigidity matroid $\mathcal{R}_{d,n}$ on $K_n$ is now the row matroid of the matrix $R(p)$ for a generic choice of $p$.
We observe here that $\mathcal{R}_{d,n}$ is independent of our choice of $p$, so long as the chosen $p$ is generic.

The structure of $\mathcal{R}_{2,n}$ has a purely combinatorial characterisation due to the following result of Pollaczek-Geiringer (later to be rediscovered by Laman \cite{Laman1970} independently).

\begin{theorem}[Pollaczek-Geiringer \cite{pollaczekgeiringer1927}]\label{thm:laman}
    A set $E \subseteq \binom{n}{2}$ is independent in $\mathcal{R}_{2,n}$ if and only if every non-empty subset $F \subseteq E$ with induced vertex set $V(F)$ satisfies the inequality
    \begin{equation}\label{eq:23sparse}
        |F| \leq 2|V(F)| - 3.
    \end{equation}
\end{theorem}

A graph $G=(V,E)$ is said to be \emph{$(2,3)$-sparse} if every non-empty subset $F \subset E$ satisfies \eqref{eq:23sparse}.
A graph $G$ with at least 2 vertices is said to be \emph{$(2,3)$-tight} if it is $(2,3)$-sparse and $|E|=2|V|-3$. 
Graver proved that such sparsity conditions hold generally in abstract rigidity matroids. 

\begin{theorem}[Graver {\cite[Thm.~5(b)]{graver1991}}]\label{thm:lamanish}
    Let $M$ be an abstract $d$-rigidity matroid on $K_n$. If $E$ is an independent set in $M$, then for all edge sets $F\subseteq E$ satisfying $|V(F)| \geq d$, one has
    \begin{equation*}
        |F|\leq d|V(F)| - \binom{d+1}{2}.
    \end{equation*}
\end{theorem}

Given two matroids $M,N$ that share a ground set, we say that $M$ is less than $N$ in the \emph{weak order} -- here denoted $M \preceq N$, with $M \prec N$ reserved for the case where $M \neq N$ -- if every independent set of $M$ is an independent set of $N$.
It follows from \Cref{thm:laman} and \Cref{thm:lamanish} that $\mathcal{R}_{2,n}$ is the maximal abstract 2-rigidity matroid (with respect to the weak order of matroids).

\begin{theorem}[Graver \cite{graver1991}]\label{thm:maximal}
    If $M$ is an abstract $d$-rigidity matroid on $K_n$ for $d \in \{1,2\}$, then $M \preceq \mathcal{R}_{d,n}$.
\end{theorem}

\Cref{thm:maximal} was recently strengthened by Tyomkyn.

\begin{theorem}[Tyomkyn \cite{tyomkyn2026planerigiditymatroids}]\label{thm:tyomkyn}
    Let $\mathcal{K}$ be an abstract 2-rigidity matroid family where $K_{3,3}$ is independent.
    Then $\mathcal{K} = \mathcal{R}_2$.
\end{theorem}

Very little is known about the structure of $\mathcal{R}_{d,n}$ when $d \geq 3$,
although some things are known about specific families of graphs.
For example,
the following two lemmas follow from Bolker and Roth's \cite{bolker1980bipartite} characterisation of dependent complete bipartite graphs in rigidity matroids.

\begin{lemma}\label{lem:bipcircrigsmall}
    For all integers $d \geq 2$ and $n \geq 2d+2$, every copy of $K_{d+1,d+1}$ is independent in $\mathcal{R}_{d,n}$.
\end{lemma}

\begin{lemma}\label{lem:bipcircrig}
    For all integers $d \geq 3$ and $n \geq 2d+4$, every copy of $K_{d+2,d+2}$ is a circuit in $\mathcal{R}_{d,n}$.
\end{lemma}

We can convert independent graphs in the $d$-rigidity matroid into independent graphs in the $(d+1)$-rigidity matroid using the \emph{coning operation}.
Here we recall that the \emph{cone} of a graph is the graph formed by adding a new vertex adjacent to all other vertices.

\begin{theorem}[Whiteley {\cite[Thm 5]{whiteley1983}}]\label{thm:rigidcone1}
    A graph is independent in $\mathcal{R}_d$ if and only its cone is independent in $\mathcal{R}_{d+1}$.
\end{theorem}

Coning also preserves circuits between rigidity matroids.

\begin{lemma}[{\cite[Lem.~5.3]{Garamvölgyi_Gortler_Jordán_2022}}]\label{lem:rigidcone2}
    A graph is a circuit in $\mathcal{R}_d$ if and only its cone is a circuit in $\mathcal{R}_{d+1}$.
\end{lemma}

\subsection{Hyperconnectivity matroids}

For some choice of positive integers $d,n$,
choose a generic vector $p = (p_i)_{i \in [n]} \in (\mathbb{R}^d)^n$.
Let $H(p)$ be the $\binom{n}{2} \times dn$ matrix with columns labelled $(i,k) \in [n] \times [d]$ ordered lexicographically and rows labelled by $\binom{n}{2}$ (with any order), where row $ij$ is of the form
\begin{equation*}
    \big[
     ~ \mathbf{0} ~ \cdots ~  \mathbf{0}  ~ \overbrace{p_j}^{i} ~ 
    \mathbf{0} ~  \cdots  ~ \mathbf{0}  ~ \overbrace{- p_i}^{j} ~ 
    \mathbf{0}  ~ \cdots  ~ \mathbf{0}  ~ 
    \big]
\end{equation*}
Kalai \cite{Kalai1985} defined the \emph{$d$-hyperconnectivity matroid} $\mathcal{H}_{d,n}$ on $K_n$ to be the row matroid of the matrix $H(p)$ for a generic choice of $p$.
Analogous to rigidity matroids, hyperconnectivity matroids are invariant under the choice of generic $p$.

Whilst the description of $\mathcal{H}_{d,n}$ is close to that of $\mathcal{R}_{d,n}$,
the two matroids do differ for $d \geq 2$.
This difference can be seen from comparing the following lemma to \Cref{lem:bipcircrigsmall}.

\begin{lemma}[Kalai {\cite[Thm 6.1]{Kalai1985}}]\label{lem:bipcirchyp}
    For all integers $d \geq 2$ and $n \geq 2d+2$, every copy of $K_{d+1,d+1}$ is a circuit in $\mathcal{H}_{d,n}$.
\end{lemma}

\subsection{Cofactor matroids}

For any $x,y \in \mathbb{R}$,
define the length-$d$ row vector
\begin{equation*}
    c_d(x,y)
    := 
    \begin{bmatrix}
        ~x^{d-1} & x^{d-2} y & \cdots &  x y^{d-2}   & y^{d-1}~
    \end{bmatrix}.
\end{equation*}
Now take a generic vector $p = (x_i,y_i)_{i \in [d]} \in (\mathbb{R}^2)^n$.
Let $C_d(p)$ be the $\binom{n}{2} \times dn$ matrix with columns labelled $(i,k) \in [n] \times [d]$ ordered lexicographically and rows labelled by $\binom{n}{2}$ (with any order), where row $ij$ is of the form
\begin{equation*}
    \big[
     ~ \mathbf{0} ~ \cdots ~  \mathbf{0}  ~ \overbrace{c_d(x_i-x_j,y_i-y_j)}^{i} ~ 
    \mathbf{0} ~  \cdots  ~ \mathbf{0}  ~ \overbrace{- c_d(x_i-x_j,y_i-y_j)}^{j} ~ 
    \mathbf{0}  ~ \cdots  ~ \mathbf{0}  ~ 
    \big]
\end{equation*}
Whiteley \cite{Whiteley1996} defined the \emph{$C_{d-1}^{d-2}$-cofactor matroid} on $K_n$ to be the row matroid of the matrix $C_d(p)$ for a generic choice of $p$.
In the spirit of \cite{bulavka2026maximalmatroidscounterexamples},
we opt to call this matroid the \emph{$d$-cofactor matroid},
and we denote the matroid by $\mathcal{C}_{d,n}$.
Analogous to rigidity matroids and hyperconnectivity matroids, cofactor matroids are invariant under the choice of generic $p$.

While the $d$-rigidity matroid is maximal when $d \leq 2$, it was proven that the 3-cofactor matroid is the unique maximal abstract 3-rigidity matroid.

\begin{theorem}[Clinch, Jackson and Tanigawa \cite{Clinch2022Abstract}]\label{thm:maximalityconj}
    If $M$ is an abstract $d$-rigidity matroid on $K_n$ for $d \in \{1,2,3\}$, then $M \preceq \mathcal{C}_{d,n}$.
\end{theorem}

It is conjectured that $\mathcal{R}_{3,n} = \mathcal{C}_{3,n}$ for all $n \geq 3$.
It is, however, known that $\mathcal{C}_{d,n} \not\preceq \mathcal{R}_{d,n}$ when $d\geq 4$,
which can be seen by comparing the following result to \Cref{lem:bipcircrig}.

\begin{lemma}[Whiteley; see {\cite[\S 11.4]{Whiteley1996}}]\label{lem:bipcirccof}
    For all integers $d \geq 4$ and $n \geq 2d+4$, every copy of $K_{d+2,d+2}$ is independent in $\mathcal{C}_{d,n}$.
\end{lemma}

\section{Construction of abstract rigidity matroids}\label{sec:main}

\subsection{Acceptable graphs}

We now wish to determine sufficient conditions on $G$ for $\mathcal{M}_G(N)$ to be an abstract rigidity matroid.
Before we do this, we require the following lemmas.

\begin{lemma}[see {\cite[Lem.~2.5.6]{GraverServatius}}]\label{lem:0-ext}
    Let $M$ be an abstract $d$-rigidity matroid on $K_n$ and let $G=(V,E)$ be a subgraph of $K_n$.
    Suppose that $G$ has a vertex $x$ of degree $k \leq d$,
    and set $E' = \{uv \in E : x \notin \{u,v\}\}$.
    Then $\rank_M(E) = \rank_M(E') + k$.
\end{lemma}

\begin{lemma}\label{lem:extendable_graph}
    For $n \geq d+1$, let $H=([n],F)$ be the graph where for any $i<j$ we have that $ij \in F$ if and only if $j-i \leq d$.
    Then $F$ is a basis of any abstract $d$-rigidity matroid on $K_n$.
\end{lemma}

\begin{proof}
    As $H$ is formed from a copy of $K_{d+1}$ by sequentially adding vertices of degree $d$,
    $F$ is a independent in any abstract $d$-rigidity matroid on $K_n$ by \Cref{lem:0-ext}.
    That $F$ is a basis now follows from the equality $|F| = dn - \binom{d+1}{2}$.
\end{proof}

We are now ready to state the following key lemma of the paper.
Here we denote edge addition and edge removal by $G+e$ and $G-e$ respectively.

\begin{lemma}\label{lem:circuit_condition}
    Let $d \geq 2$ and $n \geq d+3$ be integers, and let $N$ be a symmetric abstract $d$-rigidity matroid on $K_n$.
    Let $G=([n],E)$ be a graph satisfying the following:
    \begin{enumerate}
        \item $E$ is a basis of $N$; 
        \item $E \cup \{f\}$ is a circuit of $N$ for each $f \in E^c$; and
        \item $G-e+f$ is not isomorphic to $G$ for any $e \in E$ and $f \in E^c$.
    \end{enumerate}
    Then $\mathcal{M}_G(N)$ is a symmetric abstract $d$-rigidity matroid.
\end{lemma}

\begin{proof}
    Choose any two intersecting distinct elements $C_1,C_2 \in \Circ_{G}(N)$ and choose any $e \in C_1 \cap C_2$.
    As our choice of $C_1,C_2,e$ is arbitrary,
    $\mathcal{M}_G(N)$ is a matroid if there exists $C \in \Circ_{G}(N)$ satisfying $C \subset (C_1 \cup C_2) \setminus \{e\}$.
    First, observe that if $(C_1 \cup C_2) \setminus \{e\}$ contains a circuit $C$ of $N$,
    then either $C \in \Circ_{G}(N)$, or $C$ contains an element of $\Circ_{G}(N)$ as a subset (namely, the copy of $E$ contained in $C$).
    Thus it suffices to show that $(C_1 \cup C_2) \setminus \{e\}$ contains a circuit $C$ of $N$.

    If $C_1,C_2$ are circuits of $N$ then, since $N$ is a matroid, $(C_1 \cup C_2) \setminus \{e\}$ contains a circuit $C$ of $N$.
    Thus we may suppose without loss of generality that $C_1=E$.
    If $|C_1 \cup C_2| \geq dn - \binom{d+1}{2} +2$ then the rank of $N$ being $dn-\binom{d+1}{2}$ implies the set $(C_1 \cup C_2) \setminus\{e\}$ is dependent in $N$, and so contains a circuit $C$ of $N$.
    We now suppose for contradiction that $|C_1 \cup C_2| < dn - \binom{d+1}{2} +2$.
    Since $C_1=E$ and $C_2 \not\subset C_1$, we thus have that $|C_1 \cup C_2| = dn - \binom{d+1}{2} +1$.
    Fix $e'$ to be any element of $C_1 \setminus C_2$ and $f'$ to be the unique element in $C_2 \setminus C_1$.
    If $C_2$ is a circuit of $N$ not containing a copy of $E$, then it is the unique circuit of $N$ contained in $E \cup \{f'\}$.
    However, this implies $E \cup \{f'\} = C_2$ by assumption, which contradicts that $C_2$ doesn't contain a copy of $E$.
    If $C_2$ is a copy of $E$,
    then $C_2 = C_1 \setminus \{e' \} \cup \{f'\}$.
    However, this contradicts that $G -e'+f' \not\cong G$ for any edge $e'$ and non-edge $f'$. 
    
    The above shows that $\mathcal{M}_{G}(N)$ is a matroid, and it is symmetric by construction.
    We still need to show that $\mathcal{M}_{G}(N)$ is an abstract $d$-rigidity matroid. 
    Firstly, since $n \geq d+3$,
    every copy of $K_{d+2}$ is still a circuit in $\mathcal{M}_G(N)$.
    Let $H=([n],F)$ be the graph described in \Cref{lem:extendable_graph}.
    Here we note that, given $s$ connects the vertices $1$ and $d+1$, the graph $F \cup \{s\}$ is not a circuit of $N$ since it contains a copy of $K_{d+2}$.
    Thus $H$ cannot be isomorphic to $G$.
    By construction this implies $F$ is a basis of $\mathcal{M}_G(N)$, 
    and so $M$ has rank $dn - \binom{d+1}{2}$.
    Thus $\mathcal{M}_{G}(N)$ is an abstract $d$-rigidity matroid as required.
\end{proof}

We now define any graph satisfying all three conditions given in \Cref{lem:circuit_condition} to be \emph{$N$-acceptable}.
We note here that if $M,N$ are abstract $d$-rigidity matroids on $K_n$ satisfying $N \preceq M$,
then every $N$-acceptable graph is also $M$-acceptable;
to see this, observe that condition (i) follows from $\rank M = \rank N$, condition (ii) follows from spanning circuits of $N$ being spanning circuits of $M$, and (iii) is independent of the choice of matroid.

As the criteria for a graph to be $N$-acceptable can be difficult to work with, we will sometimes opt to work with the following softening of the third property.
Here we denote the degree of a vertex $v$ in $G$ by $\deg_G(v)$.

\begin{lemma}\label{lem:easy_acceptable}
    Let $G=([n],E)$ be a graph where $|\deg_G(u) - \deg_G(v)| \neq 1$ for all $u,v \in [n]$.
    Then $G-e+f$ is not isomorphic to $G$ for any $e \in E$ and $f \in E^c$.
\end{lemma}

\begin{proof}
    Choose any $e=uv \in E$ and any $f \in E^c$.
    Then, without loss of generality, $\deg_{G-e+f}(u) = \deg_G(u) - 1$.
    Our assumption now implies that $\deg_{G-e+f}(u) \neq \deg_G(x)$ for all $x \in [n]$.
    As the degree sequence of $G-e+f$ cannot match the degree sequence of $G$, the pair $G$ and $G-e+f$ are not isomorphic.
\end{proof}

Coning also preserves admissibility in some specific cases.

\begin{lemma}\label{lem:acceptable_cone}
    Let $d \geq 2$.
    If $G$ is $\mathcal{R}_{d,n}$-acceptable and contains no vertex of degree $n-2$,
    then the cone of $G$ is $\mathcal{R}_{d+1,n+1}$-acceptable and contains no vertex of degree $n-1$.
\end{lemma}

\begin{proof}
    Let $G'=([n+1],E')$ be the cone of $G$.
    By \Cref{thm:rigidcone1} and \Cref{lem:rigidcone2},
    the edge set $E'$ is a basis of $\mathcal{R}_{n+1,d+1}$ and the edge set $E'\cup \{f\}$ is a circuit of $\mathcal{R}_{n+1,d+1}$ for all $f \in \binom{[n+1]}{2} \setminus E' = \binom{[n]}{2} \setminus E$.
    Additionally, as no vertex of $G$ has degree $n-2$, no vertex of $G'$ has degree $n-1$.

    Suppose now, for the sake of contradiction, that $G'-e+f$ is isomorphic to $G'$ for some edge $e \in E'$ and non-edge $f \in \binom{[n+1]}{2} \setminus E' = \binom{[n]}{2} \setminus E$.
    As no vertex of $G$ has degree $n-2$, and thus no vertex of $G'$ has degree $n-1$, 
    the edge $e$ cannot be adjacent to the cone vertex $n+1$.
    However, this now implies that $G-e+f$ is isomorphic to $G$, which contradicts that $G$ is $\mathcal{R}_{d,n}$-acceptable.
\end{proof}

\begin{remark}\label{rem:haccept}
    It is tempting to posit that a graph $G=([n],E)$ with $n \geq 4$ will be $\mathcal{H}_{2,n}$-acceptable if it is a base of the 2-hyperconnectivity matroid, contains no $(2,3)$-tight proper subgraphs with 4 or more vertices, and $G-e+f$ is not isomorphic to $G$ for any $e \in E$ and $f \in E^c$.
    However, $(2,3)$-tight graphs can be dependent in the 2-hyperconnectivity matroid (for example, $K_{3,3}$),
    and so any such graph can still fail the second condition outlined in \Cref{lem:circuit_condition}.
\end{remark}

\subsection{Constructing even-order \texorpdfstring{$\mathcal{R}_{2,n}$}{R2n}-acceptable graphs}\label{subsec:even}

By \Cref{thm:laman}, a graph $G=([n],E)$ with $n \geq 4$ is $\mathcal{R}_{2,n}$-acceptable if and only if it is $(2,3)$-tight, contains no $(2,3)$-tight proper subgraphs with 4 or more vertices, and $G-e+f$ is not isomorphic to $G$ for any $e \in E$ and $f \in E^c$.
Our aim now is to construct an infinite family of $\mathcal{R}_{2,n}$-acceptable graphs with no vertex of degree $n-2$.
With this, we can apply \Cref{lem:acceptable_cone} to prove \Cref{mainthm}.
We will first construct even-order $\mathcal{R}_{2,n}$-acceptable graphs.

For the following lemma, we recall that a graph is a \emph{Minimally Isostatic Graph (MIG)} if it is $(2,3)$-tight and contains no $(2,3)$-tight proper subgraphs with 3 or more vertices.
We will also opt to denote our bipartite graphs by $G=(X,Y,E)$, with $(X,Y)$ denoting a specific bipartite partition of the vertex set.

\begin{lemma}\label{lem:1-ext}
    Let $G=(X,Y,E)$ be a bipartite graph satisfying the following conditions:
    \begin{enumerate}
        \item $G$ is a MIG;
        \item $G$ has exactly two vertices $x^*,y^*$ of degree $\Delta_G$, and they are contained in $X$ and $Y$ respectively; and
        \item $|\deg_G(u) - \deg_G(v)| \neq 1$ for all $u,v \in X \cup Y$.
    \end{enumerate}
    Then there exists a bipartite graph with $|X|+|Y|+2$ vertices satisfying the above conditions.
\end{lemma}

\begin{proof}
    First, observe that $G$ must have minimum degree 3, since it is a MIG with at least 4 vertices (as the single MIG with three vertices, $K_3$, is not bipartite).
    Choose any edge $xy$ with $x \in X \setminus \{x^*\}$ and $y \in Y \setminus\{y^*\}$.
    (Such an edge must exist, as otherwise $G$ is a forest and thus not a MIG.)
    We now perform two consecutive 1-extensions to form a new graph $G' = (X',Y',E')$ as follows:
    (1) delete the edge $xy$, then add the vertex $a$ and edges $ax,ay,ay^*$;
    (2) delete the edge $ax$, then add the vertex $b$ and edges $ab,bx,bx^*$.
    It is easy to check that the new graph $G'$ is bipartite and satisfies conditions (ii) and (iii).
    
    Suppose for the sake of contradiction that $G'$ is not a MIG.
    Then there exist $A \subseteq X$ and $B \subseteq Y$ such that $A \cup B \neq X \cup Y$ and the subgraph $H' \subset G'$ induced by $A \cup B$ is $(2,3)$-tight.
    We suppose now that from the family of all such pairs $(A,B)$ inducing a $(2,3)$-tight subgraph with at least 3 vertices, we chose $(A,B)$ to minimise $|A|+|B|$.
    With this, we can assume that $H'$ has minimum degree 3;
    if not, then we could always choose a strictly smaller subgraph, with the knowledge that, as $H'$ is bipartite, this process cannot reduce the graph to a single edge.
    If $a \notin A$ and $b \notin B$, then $H'$ is contained in $G$, contradicting that $G$ is a MIG.
    Suppose, without loss of generality, that $a \in A$.
    Since the minimal degree of $H'$ is 3,
    $B$ must contain all of the neighbours of $a$, namely $b,y,y^*$.
    We now apply the same logic to see that $x,x^* \in A$.
    Now fix the subgraph $H \subset G$ induced by the vertex set $(A\setminus\{a\}) \cup (B \setminus \{b\})$.
    As $H'$ is $(2,3)$-tight, the graph $H$ must also be $(2,3)$-tight.
    Since $G$ is a MIG, we have that $H$ has two vertices, contradicting that $\{x,y,x^*,y^*\} \subseteq (A\setminus\{a\}) \cup (B \setminus \{b\})$.
\end{proof}

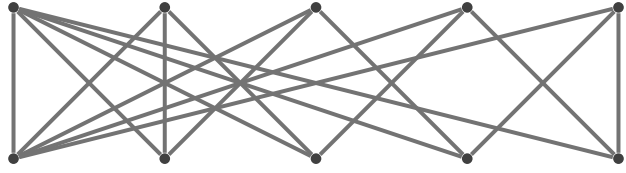
\begin{figure}
    \centering
    \begin{tikzpicture}[scale=2]
			\node[vertex] (1b) at (0,0) {};
			\node[vertex] (2b) at (1,0) {};
			\node[vertex] (3b) at (2,0) {};
			\node[vertex] (4b) at (3,0) {};
			\node[vertex] (5b) at (4,0) {};

			\node[vertex] (1t) at (0,1) {};
			\node[vertex] (2t) at (1,1) {};
			\node[vertex] (3t) at (2,1) {};
			\node[vertex] (4t) at (3,1) {};
			\node[vertex] (5t) at (4,1) {};

			\draw[edge] (1b)edge(1t);
			\draw[edge] (1b)edge(2t);
			\draw[edge] (1b)edge(3t);
			\draw[edge] (1b)edge(4t);
			\draw[edge] (1b)edge(5t);
			\draw[edge] (2b)edge(1t);
			\draw[edge] (3b)edge(1t);
            \draw[edge] (4b)edge(1t);
			\draw[edge] (5b)edge(1t);
            
			\draw[edge] (2b)edge(2t);
			\draw[edge] (2b)edge(3t);
			\draw[edge] (3b)edge(2t);

			\draw[edge] (3b)edge(4t);
			\draw[edge] (4b)edge(3t);

			\draw[edge] (5b)edge(4t);
			\draw[edge] (4b)edge(5t);
            
			\draw[edge] (5b)edge(5t);
		\end{tikzpicture}
    \caption{A bipartite graph satisfying conditions (i), (ii) and (iii) in \Cref{lem:1-ext}.}
    \label{fig:bipartite}
\end{figure}

The bipartite graph in \Cref{fig:bipartite} satisfies conditions (i), (ii) and (iii) of \Cref{lem:1-ext}.
Applying \Cref{lem:1-ext} now generates an infinite family of bipartite $\mathcal{R}_{2,n}$-acceptable graphs; one for each $n = 2k + 8$ with $k$ being a positive integer.
Since each graph is bipartite,
they cannot contain a vertex of degree $n-2$;
this is as the only bipartite graph with minimum degree at least 2 with a vertex of degree $n-2$ is $K_{2,n}$, which is not $(2,3)$-tight (and hence not a basis of $\mathcal{R}_{2,n}$).

\subsection{Constructing odd-order \texorpdfstring{$\mathcal{R}_{2,n}$}{R2n}-acceptable graphs}\label{subsec:odd}

We now construct odd-order $\mathcal{R}_{2,n}$-acceptable graphs.

Let $n\geq 3$. We define $G_{2n + 1}$ to be the following graph with vertex set $[2n+1]$ and edges described on the right (with edge colouring added for visual clarity):
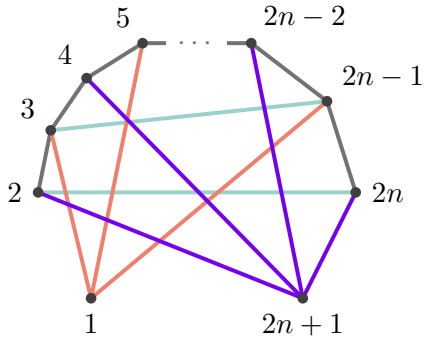
\begin{center}
    
\begin{tikzpicture}[scale=0.7]
\begin{scope}
\coordinate (v1)    at (-2,-2);
\coordinate (vlast) at ( 2,-2);

\foreach \name/\angle in {
    2/180, 3/157, 4/134, 5/110,
    a/70, b/35, c/0
}{
    \coordinate (v\name) at (\angle:3);
}

\draw[edge] (v2) -- (v3) -- (v4) -- (v5);
\draw[edge] (va) -- (vb) -- (vc);

\draw[edge] (v5) --
    node[fill=white, inner sep=4pt] {$\cdots$}
    (va);

\draw[bedge] (v2) -- (vc);   
\draw[bedge] (v3) -- (vb);   

\foreach \target in {3,5,b}{
    \draw[oedge] (v1) -- (v\target);
}

\foreach \target in {2,4,a,c}{
    \draw[pedge] (vlast) -- (v\target);
}

\node[vertex, label=below:{$1$}]       at (v1)    {};
\node[vertex, label=below:{$2n+1$}]    at (vlast) {};

\foreach \name/\angle/\lab in {
    2/180/2,
    3/157/3,
    4/134/4,
    5/110/5,
    a/70/{2n-2},
    b/35/{2n-1},
    c/0/{2n}
}{
    \node[vertex, label={\angle:{$\lab$}}] at (v\name) {};
}
\end{scope}
\begin{scope}[xshift=300]
    \node[align=left] at (0,0) {
    \textcolor{colgraphe}{$(i,i+1)$ for $ i \in \{2, \dots, 2n-1\}$}  \\ \textcolor{myblue}{$(2, 2n), (3, 2n-1)$} \\ \textcolor{myorange}{$(1, i)$ for $i \in \{3, 5, 7, \ldots, 2n-1\}$} \\ \textcolor{mypurple}{$(2n+1, i)$ for $i \in \{2, 4, 6, \ldots, 2n\}$}};
\end{scope}
\end{tikzpicture}
\end{center}

We now prove that each $G_{2n+1}$ is $\mathcal{R}_{2,2n+1}$-acceptable using the following two lemmas:

\begin{lemma}\label{lem:specialodd}
    For $n\geq 6$, the graph $G_{2n+1}$ is $(2,3)$-tight and contains no proper $(2,3)$-tight subgraph with 4 or more vertices.
\end{lemma}

\begin{proof}
    Suppose for contradiction that $G_{2n+1}$ does not have the stated properties. Choose any proper non-trivial subgraph $H$ which is $(2,3)$-tight and contains at least 4 vertices;
    additionally, we may suppose that we chose $H$ to have as few vertices as possible.
    As $H$ is $(2,3)$-tight and not a single edge, it has minimal degree either 2 or 3.
    As removing a vertex of degree 2 preserves $(2,3)$-sparsity, either $H$ has minimal degree 3 or $H$ has exactly 4 vertices.
    Since $G_{2n+1}$ does not contain two triangles sharing an edge, $H$ must have minimal degree 3.
    As both $H$ and $G_{2n+1}$ are $(2,3)$-tight, $H$ is an induced subgraph of $G_{2n+1}$.

    First, we consider the case when $H$ contains no vertex in $\{4,5, \dots, 2n-2\}$,
    i.e.~the vertex set of $H$ is contained in $\{1,2,3,2n-1,2n,2n+1\}$.
    This implies that $H$ has a vertex of degree less than $2$:
    either $H$ contains one of $1,2n+1$ as a degree 2 vertex,
    or $H$ does not contain either of $1,2n+1$ and so contains one of $\{2,3,2n-1,2n\}$ as a degree 2 vertex.
    Thus $H$ contains a vertex $x \in\{4,5, \dots, 2n-2\}$.
    
    Since $H$ has minimal degree 3, $H$ must contain each neighbour of $x$. By repeating this argument, we see that $H$ must contain the vertices $4,5,\ldots, 2n-2$, as well as their neighbours.
    In particular, $H$ contains the neighbours of $4,5,2n-1$, and so $H$ contains the vertices $1,3,2n-1,2n+1$.
    Thus $H$ is formed from $G$ by either deleting $2$, deleting $2n$, or deleting both $2$ and $2n$.
    However, it is easy to check that each such graph has too few edges to be $(2,3)$-tight.
\end{proof}

\begin{lemma}
    For $n\geq 6$, and any edge $e\in E(G_{2n+1})$ and non-edge $f\in \binom{2n +1}{2} \setminus E(G_{2n+1})$, the graph $G_{2n+1} -e + f$ is not isomorphic to $G_{2n +1}$.
\end{lemma}
\begin{proof}
Fix $G=G_{2n+1}$ for brevity.
Suppose for contradiction that $G -e + f$ is isomorphic to $G$ for some choice of edge $e = \{u_e,v_e\}$ and non-edge $f=\{u_f,v_f\}$ where $v_e \neq v_f$.
Set $\phi:[2n+1] \rightarrow [2n+1]$ to be the graph isomorphism from $G$ to $G-e+f$.
The vertices of $G$ can be separated into four families depending on degree:
vertex $2n+1$ has degree $n$,
vertex $1$ has degree $n-1$,
vertices $3$ and $2n-1$ have degree $4$,
and all other vertices have degree $3$. 
As $n \geq 6$,
we have that $v_e \in \{3,2n-1,2n+1\}$,
or else the degree sequence of $G-e+f$ will not match the degree sequence of $G$.
We now go through the possible cases for $n\geq 6$. 

First, we suppose that $v_e=2n+1$.
So that the degree sequences agree, we must have (without loss of generality) that $v_f = 1$ and $u_e=u_f=u$ for some $u \in \{2,4,6,\ldots,2n-2,2n\}$.
However, no matter what our choice of $u$ is, $G-e+f$ contains two triangles sharing an edge which $G$ does not:
if $u=2$ then we have edge-sharing triangles $(1,2,3)$ and $(1,3,2n-1)$, and if $u >2$ then we have edge-sharing triangles $(1,u-1,u)$ and $(1,u,u+1)$.
Thus $v_e \neq 2n+1$, which in turn implies $v_f \neq 1$.
Additionally, $u_e \neq 2n+1$:
if $u_e \neq u_f$ then switching the labels of $u_e,v_e$ and applying the previous argument shows that $u_e \neq 2n+1$;
and if $u_e = u_f=2n+1$ then $v_e$ has degree 3 in $G$,
but this would force $\phi(v_e)$ to have degree 2 in $G-e+f$.

We now have that $v_e \in \{3,2n-1\}$.
It is clear that $\phi(2n+1) = 2n+1$ as $2n+1$ has the same degree in both $G$ and $G-e+f$.
By a similar argument, $\phi(1) = 1$.
So that $1$ and $2n+1$ have the correct neighbours in $G-e+f$, $\phi$ must preserve parity, in that it maps odd vertices to odd vertices and even vertices to even vertices.
As the only adjacent even vertices are $2,2n$,
we have that $\phi(2) \in \{2,2n\}$ and $\phi(2n) \in \{2,2n\}$.
If $u_e \neq u_f$, then (by relabelling) we see that $u_e \neq 2n+1$ also, and so $u_e \in \{3,2n-1\}$.
Since $2,2n$ are the unique vertices contained in a triangle with $2n+1$ and $\{u_e,v_e\} \neq \{2,2n\}$,
we have $\{\phi(2),\phi(2n)\} = \{2,2n\}$.
We now observe that $G$ has the following automorphism:
\begin{equation*}
    i \mapsto 
    \begin{cases}
        i &\text{if } i \in \{1,2n+1\},\\
        2n+2-i &\text{otherwise}.
    \end{cases}
\end{equation*}
Thus we may suppose that $\phi(2) = 2$ and $\phi(2n) = 2n$.
Since $u_e \neq 2n+1$,
we now have 8 possible cases to check, each ordered $(u_e,v_e)$:
\begin{equation*}
    (1,3), ~ (2,3), ~ (4,3),~ (2n-1,3),~ (1,2n-1), ~ (3,2n-1),~ (2n-2,2n-1), ~ (2n,2n-1).
\end{equation*}
We now check each possibility below:
\begin{itemize}
    \item Either $e=(1,3)$ and $e=(1,2n-1)$:
    To preserve the degree sequence, $u_e=u_f =1$.
    Thus the vertex $v_f$ is even. The vertex $v_f$ is now forced to be adjacent to both $1$ and $2n+1$ in $G-e+f$, but no vertex in $G$ is adjacent to both vertices.
    \item Either $e=(2,3)$ or $e=(4,3)$: We consider first the case where $e=(2,3)$. To preserve the degree sequence one has $u_e=u_f=2$ and $v_f$ has degree 3 in $G$.
    As the degree 4 vertices in $G$ are adjacent, $v_f$ must be adjacent to $2n-1$. The only possibilities are then $v_f=2n-2$ or $v_f =2n$. The edge $(2, 2n)$ is already an edge of $G$, and thus $f\neq (2,2n)$, while $f=(2, 2n-2)$ would create an additional triangle containing vertex $2n+1$ in $G$. In the case where $e=(4,3)$, one similarly derives from degree considerations and since $f$ must be adjacent to $2n-1$ that $f$ has to be $(4,2n-2)$ or $(4, 2n)$. Both of these edges would create an additional triangle containing $2n+1$ in the graph, which is impossible.
    \item Either $e=(2n-2,2n-1)$ or $e=(2n,2n-1)$. We first consider $e=(2n-2, 2n-1)$. By degree considerations, we must have $f=(2n-2, v_f)$ for some $v_f$, and $v_f$ must be adjacent to $3$. The only possibilities are $v_f=2$ or $v_f=3$, but these cases would create an additional triangle with vertex $2n+1$. For the edge $e=(2n, 2n-1)$, by degree considerations and adjacency in the graph one needs $f=(2n, v_f)$, and $v_f$ must be $2$ or $4$. The edge $2n, 2$ is already present in the graph, so this cannot be $f$, and the edge $(2n, 4)$ would create an additional triangle with vertex $2n+1$. 
    \item $\{u_e,v_f\} = \{3,2n-1\}$:
    Here we observe that $G$ contains a 4-cycle $(2,2n,2n-1,3)$, and so $(\phi(2),\phi(2n),\phi(2n-1),\phi(3)\}$ is a 4-cycle in $G-e+f$.
    The neighbours of $2$ in $G-e+f$ are $\{3,2n,2n+1\}$ and the $2n$ in $G-e+f$ are $\{2,2n-1,2n+1\}$. Since $\phi(2) = 2$, $\phi(2n) = 2n$ and $\phi(2n+1) = 2n+1$, we must have that $\phi(3) = 3$ and $\phi(2n-1)$.
    However, since $\{\phi(3),\phi(2n-1)\} = \{u_f,v_f\}$, this now implies $e=f$.
\end{itemize}
The cases all now provide the desired contradiction.
\end{proof}

Combining the results given in \Cref{subsec:even} and \Cref{subsec:odd} now gives the following result:

\begin{lemma}\label{lem:oddneven}
    For each $n \geq 13$,
    there exists an $\mathcal{R}_{2,n}$-acceptable graph with no vertex of degree $n-2$.
\end{lemma}

Interestingly, each graph $G_{2n+1}$ is base of the 2-hyperconnectivity matroid.

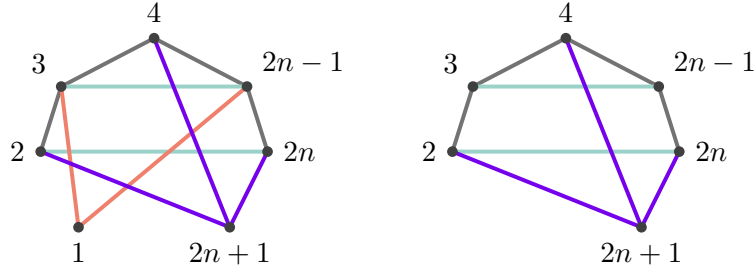
\begin{figure}
\begin{center}
\begin{tikzpicture}[scale=0.5]
\coordinate (v1)    at (-2,-2);
\coordinate (vlast) at ( 2,-2);

\foreach \name/\angle in {
    2/180, 3/145, 4/90, b/35, c/0
}{
    \coordinate (v\name) at (\angle:3);
}

\draw[edge] (v2) -- (v3) -- (v4) -- (vb) -- (vc);

\draw[bedge] (v2) -- (vc);   
\draw[bedge] (v3) -- (vb);   

\foreach \target in {3,b}{
    \draw[oedge] (v1) -- (v\target);
}

\foreach \target in {2,4,c}{
    \draw[pedge] (vlast) -- (v\target);
}

\node[vertex, label=below:{$1$}]       at (v1)    {};
\node[vertex, label=below:{$2n+1$}]    at (vlast) {};

\foreach \name/\angle/\lab in {
    2/180/2,
    3/145/3,
    4/90/4,
    b/35/{2n-1},
    c/0/{2n}
}{
    \node[vertex, label={\angle:{$\lab$}}] at (v\name) {};
}
\end{tikzpicture}\qquad
\begin{tikzpicture}[scale=0.5]
\coordinate (vlast) at ( 2,-2);

\foreach \name/\angle in {
    2/180, 3/145, 4/90, b/35, c/0
}{
    \coordinate (v\name) at (\angle:3);
}

\draw[edge] (v2) -- (v3) -- (v4) -- (vb) -- (vc);

\draw[bedge] (v2) -- (vc);   
\draw[bedge] (v3) -- (vb);   

\foreach \target in {2,4,c}{
    \draw[pedge] (vlast) -- (v\target);
}

\node[vertex, label=below:{$2n+1$}]    at (vlast) {};

\foreach \name/\angle/\lab in {
    2/180/2,
    3/145/3,
    4/90/4,
    b/35/{2n-1},
    c/0/{2n}
}{
    \node[vertex, label={\angle:{$\lab$}}] at (v\name) {};
}
\end{tikzpicture}
\end{center}
\caption{(Left) The graph $G_{2n+1}$ for $n=3$. (Right) The graph $H$ formed from $G_7 = G_{2n+1}$ by deleting the vertex 1.}
\end{figure}

\begin{proposition}\label{prop:hypconn}
    For $n \geq 3$,
    the graph $G_{2n+1}$ is a base of the 2-hyperconnectivity matroid.
\end{proposition}

\begin{proof}
    First, fix $n=3$.
    As can be seen in from the left-hand graph in \Cref{fig:bipartite},
    the vertex 1 has degree 2 in $G_{7}$.
    Let $H$ be the graph formed from $G_7$ by deleting the vertex 1.
    Since $H$ a cubic graph not isomorphic to $K_4$ or $K_{3,3}$, it is a base of the 2-hyperconnectivity matroid (see \cite[Theorem 1.8]{tyomkyn2026planerigiditymatroids}).
    As independence in the $d$-hyperconnectivity matroid is preserved when a new degree $d$ vertex is appended to the graph (\cite[Lemma 2.1(b)]{garamvölgyi2026}),
    $G_7$ is a base of the 2-hyperconnectivity matroid.

    Now fix $n >3$ and suppose that the result holds for all $k \leq n$.
    Here we observe that $G_{2n+1}$ can be formed from $G_{2n-1}$ by the following operation:
    delete the edge $(3,4)$, then add two new vertices $a,b$ and edges $(a,b)$, $(a,4)$, $(a,2n+1)$, $(b,3)$, $(b,1)$.
    As this operation also preserves independence in the 2-hyperconnectivity matroid (\cite[Lemma 2.1(b)]{garamvölgyi2026}),
    it follows that $G_{2n+1}$ is a base of the 2-hyperconnectivity matroid.
\end{proof}

\subsection{Proof of \texorpdfstring{\Cref{mainthm}}{main theorem} and corollaries}

We are now ready to prove \Cref{mainthm}, \Cref{maincor} and \Cref{maincor2}.

\begin{proof}[Proof of \Cref{mainthm}]
    By \Cref{lem:oddneven} and \Cref{lem:acceptable_cone},
    for each $d \geq 2$ and each $n \geq  d +11$ there exists a graph that is $\mathcal{R}_{d,n}$-acceptable.
    The result now follows from applying \Cref{lem:circuit_condition}.
\end{proof}

\begin{proof}[Proof of \Cref{maincor}]
    Fix $d \geq 2$.
    For each $n \geq d+11$,
    fix $G_n$ to be the $n$-vertex graph given by \Cref{mainthm} for our chosen integers $d,n$,
    and fix the symmetric abstract $d$-rigidity matroid $M_n := \mathcal{M}_{G_n}(\mathcal{R}_{d,n})$.
    By design, we have $M_n \prec \mathcal{R}_{d,n}$ for each $n \geq d+ 11$.
    If $d \in \{2,3\}$ then, as $n >6$, we have $M_n \prec \mathcal{R}_{d,n} \preceq \mathcal{C}_{d,n}$ (\Cref{thm:maximalityconj}), and so $M_n \neq \mathcal{C}_{d,n}$.
    The restriction of each $M_n$ to $K_{n-1}$ is exactly $\mathcal{R}_{d,n}$.
    This implies two things for each $n \geq d+11$:
    (i) if $n > 2d+2$,
    then every copy of $K_{d+1,d+1}$ is independent in $M_n$ (\Cref{lem:bipcircrigsmall}); and
    (ii) if $n > 2d+4$,
    then every copy of $K_{d+2,d+2}$ is dependent in $M_n$ (\Cref{lem:bipcircrig}).
    The first implication means that $M_n \neq \mathcal{H}_{d,n}$ when $n >2d+2$ (\Cref{lem:bipcirchyp}), and the second implication means that $M_n \neq \mathcal{C}_{d,n}$ when $n >2d+4$ and $d \geq 4$ (\Cref{lem:bipcirccof}).
    Thus if $n \geq 2d+9 \geq  \max\{ d+11, 2d+5\}$, then $M_n \notin \{\mathcal{R}_{d,n}, \mathcal{H}_{d,n},\mathcal{C}_{d,n}\}$.
\end{proof}

\begin{proof}[Proof of \Cref{maincor2}]
    Fix $n \geq 13$.
    By \Cref{mainthm},
    there exists a graph $G=([n],E)$ where $E$ is a basis of $\mathcal{R}_{2,n}$ and $\mathcal{M}_{G}(\mathcal{R}_{2,n})$ is a symmetric abstract 2-rigidity matroid.
    Suppose for the sake of contradiction that $\mathcal{M}_{G}(\mathcal{R}_{2,n}) = \mathcal{K}_n$ for some abstract 2-rigidity matroid family $\mathcal{K} = (\mathcal{K}_n)_{n \geq d}$.
    As $n >6$, $K_{3,3}$ is independent in $\mathcal{M}_{G}(\mathcal{R}_{2,n})$, and thus $K_{3,3}$ is also independent $\mathcal{K}$.
    By \Cref{thm:tyomkyn}, $\mathcal{K} = \mathcal{R}_2$.
    However, this now implies that $E$ is independent in $\mathcal{K}_n = \mathcal{M}_{G}(\mathcal{R}_{2,n})$,
    which contradicts our construction.
\end{proof}

\section{Open problems}\label{sec:further}

\subsection{Extending main results to hyperconnectivity and cofactor matroids}

The values of $n$ for which \Cref{mainthm}, \Cref{maincor} and \Cref{maincor2} have not been optimised, so it is likely a tight result can easily be stated in each case with respect to $n$.
As observed by Tyomkin \cite{tyomkyn2026planerigiditymatroids},
there are no $\mathcal{R}_{2,7}$-acceptable graphs;
of the 70 $(2,3)$-tight graphs with 7 vertices \cite{capco_2018}, exactly 2 contain no $(2,3)$-tight proper subgraph with four or more vertices\footnote{The two graphs are 56155 and 35507 on House of Graphs \cite{houseofgraphs}: \href{https://houseofgraphs.org/}{https://houseofgraphs.org/}.} and both fail condition (iii) of \Cref{lem:circuit_condition}.
The remaining cases to check for $d=2$ are for $n \in \{8,9,11\}$.

It seems plausible that \Cref{mainthm} remains true if we switch $\mathcal{R}_d$ with either $\mathcal{H}_d$ or $\mathcal{C}_d$.
There are, however, a few road-blocks to doing so:
\begin{itemize}
    \item \textbf{Hyperconnectivity matroids:} It is known that independent sets and circuits of the hyperconnectivity matroid are preserved under coning \cite[Cor.~5.2]{Kalai1985}\footnote{While Kalai's paper states without proof that coning preserves circuits in the hyperconnectivity matroid, the proof of \Cref{lem:rigidcone2} given in \cite[Lem.~5.3]{Garamvölgyi_Gortler_Jordán_2022} can be easily adapted to prove this result. An alternative proof can be found in \cite[Prop.~7(b)]{bulavka2026maximalmatroidscounterexamples}.}, thus \Cref{lem:acceptable_cone} can easily extended to $\mathcal{H}_d$. 
    There are two issues with the constructed graphs described in the proof of \Cref{mainthm}.
    Firstly, the even-order constructed graphs in \Cref{mainthm} are all bipartite graphs, and thus cannot be bases of the 2-hyperconnectivity matroid on $K_n$ (see \cite[Thm 6.1]{Kalai1985}).
    Secondly, although the odd-order constructed graphs given in \Cref{subsec:odd} are bases of 2-hyperconnectivity matroid (\Cref{prop:hypconn}),
    it is unclear if said graphs are $\mathcal{H}_{2,n}$-acceptable since the circuits of 2-hyperconnectivity matroid have no simple characterisation.
    \item \textbf{Cofactor matroids:} 
    Unlike for the 2-hyperconnectivity matroid,
    the constructed graphs given in \Cref{mainthm} are bases of the 2-cofactor matroid on $K_n$.
    However, it is currently not known whether coning preserves circuits for cofactor matroids (although coning does preserve independence for cofactor matroids \cite[Thm 11.3.3]{Whiteley1996}), and so we cannot extend the result into values of $d$ where $\mathcal{R}_{d,n} \neq \mathcal{C}_{d,n}$.
    It is worth noting here that the proof of \Cref{lem:rigidcone2} given in \cite[Lem.~5.3]{Garamvölgyi_Gortler_Jordán_2022} cannot directly be adapted to the cofactor case.
\end{itemize}

We conjecture that \Cref{maincor2} is true in all dimensions higher than 1.

\begin{conjecture}
    For each $d \geq 2$, there are infinitely many positive integers $n$ for which there exists a symmetric abstract $d$-dimensional rigidity matroid on $K_n$ that is not contained in any abstract $d$-rigidity matroid family.
\end{conjecture}

\subsection{Realisability of constructed matroids}

It was recently proven by Tyomkyn \cite{tyomkyn2026linearmatroidproductssynthetic} that every linearly representable abstract 2-rigidity matroid is the row matroid of some matrix $R(p)$ (as defined in \Cref{subsec:rigiditymatroid}) with possibly non-generic choice of $p$.
Using this, we can prove the following.

\begin{proposition}\label{prop:prism}
    Given $G$ is the triangular prism pictured in \Cref{fig:3prism}, the matroid $\mathcal{M}_{G}(\mathcal{R}_{2,6})$ is not linearly representable over any field of characteristic 0.
\end{proposition}

\begin{proof}
    Let $G$ be labelled as in \Cref{fig:3prism}.
    Suppose for contradiction that $\mathcal{M}_{G}(\mathcal{R}_{2,6})$ is linearly representable over the characteristic 0 field $\mathbb{F}$.
    By \cite[Thm 1.5]{tyomkyn2026linearmatroidproductssynthetic},
    there exists a vector $p \in (\mathbb{F}^2)^6$ so that $\mathcal{M}_{G}(\mathcal{F}_{2,6})$ is the row matroid of $R(p)$ (here we extend our definition of $R(p)$ to allow for entries in $\mathbb{F}$).
    As $K_3$ is independent in any abstract 2-rigidity matroid,
    the points $p(i),p(j),p(k)$ are not collinear for any three distinct vertices $i,j,k \in [6]$.
    It is a classical result that, if both the triangles $(a,b,c)$, $(a',b',c')$ are not collinear, then $G$ is dependent in the row matroid of $R(p)$ if and only if the three lines $[p(a),p(a')]$, $[p(b),p(c')]$, $[p(b'),p(c)]$ are either parallel or intersect at a point;
    see for example \cite[Tab.~1.3]{whitewhiteley}.
    As projective transformations of $p$ do not effect the row matroid of $R(p)$ (see \cite[Thm 1]{Izmestiev2009}),
    we may suppose that 
    \begin{gather*}
        p(a) = (0,0), \qquad p(a') = (1,0), \qquad p(b) = (0,1), \qquad p(b') = (1,1), \\  p(c) = (x , y), \qquad p(c') = (x', y).
    \end{gather*}
    So that we do not have a collinear quadruple of points,
    we must have that $y \notin \{0,1\}$.

    For any permutation $\sigma$ of $[n]$, the graph $\sigma G$ is a circuit of the row matroid of $R(p)$.
    Added to the exact characterisation of when a realisation of the triangular prism is dependent,
    we obtain the following geometric constraints for $p$ from the following permutations:
    \begin{itemize}
        \item Switch $a \leftrightarrow b$ and $a' \leftrightarrow b'$: the point $(1/2,1/2)$ (i.e.~the intersection of $[p(a),p(b')]$, $[p(b),p(a')]$) lies on the line $[p(c),p(c')]$.
        Thus $y=1/2$.
        \item Switch $b \leftrightarrow c$ and $b' \leftrightarrow c'$: the lines $[p(a),p(a')]$, $[p(b),p(c')]$ and $[p(c),p(b')]$ intersect at a point $(t,0)$ for some $t \in \mathbb{R}$, or the lines are parallel.
        The lines $[p(b),p(c')]$ and $[p(c),p(b')]$ intersect at the unique point
        \begin{equation*}
            \frac{1}{x' - x +1} ( x', y+x'-x)
        \end{equation*}
        if and only if $x' - x +1 \neq 0$,
        otherwise they are parallel.
        However, if $x' = x-1$, then the lines $[p(b),p(c')]$ and $[p(c),p(b')]$ intersect $[p(a),p(a')]$ at 
        \begin{equation*}
            \left(\frac{x-1}{1-y}~,~0 \right), \qquad  \left(1+\frac{x-1}{1-y}~,~0 \right)
        \end{equation*}
        respectively, contradicting that the three lines are parallel.
        Thus $x' - x +1 \neq 0$ and (as the second coordinate of the intersection must be 0) $y = x- x'$.
        \item Switch $a \leftrightarrow c$ and $a' \leftrightarrow c'$: the lines $[p(a),p(c')]$, $[p(b),p(b')]$ and $[p(c),p(a')]$ intersect at a point $(t,1)$ for some $t \in \mathbb{R}$,
        or the lines are parallel.
        We now compute that the lines intersect at the point
        \begin{equation*}
            \frac{1}{x' - x +1} ( x', y)
        \end{equation*}
        if it exists.
        As $x'-x+1 = 1/2 \neq 0$,
        the intersection exists and is exactly the point $((2x-1)/3,1/3)$.
        However, this now contradicts that the second coordinate must be 1.
    \end{itemize}
    As the geometric constraints are not consistent, we obtain our desired contradiction.
\end{proof}

Because of this, we believe the following holds.

\begin{conjecture}
    Let $G=([n],E)$ be a graph where $\mathcal{M}_G(\mathcal{R}_{2,n})$ is an abstract 2-rigidity matroid.
    If $G$ is not isomorphic to $K_{3,3}$,
    then $\mathcal{M}_G(\mathcal{R}_{2,n})$ is not linearly realisable over any field of characteristic 0.
\end{conjecture}

\addcontentsline{toc}{section}{Acknowledgements}
\section*{Acknowledgements}

S.\,D.\ was supported by the FWO grants G0F5921N (Odysseus) and G023721N, and by the KU Leuven grant iBOF/23/064.
The authors would like to thank D\'{a}niel Garamv\"{o}lgyi for his insights into abstract rigidity matroids and cofactor matroid coning, as well as his helpful comments on an earlier draft of the paper.
The authors would also like to thank Mykhaylo Tyomkyn for his helpful comments on an earlier draft of the paper.

\bibliographystyle{plainurl}
\bibliography{ref}

\end{document}